\documentclass[12pt]{amsart}

\usepackage{amsmath,amssymb,euscript,mathrsfs, eufrak,tikz}
\usepackage[all]{xy}
\usepackage{enumerate}

\pushQED{\qed}

\newcommand{\define}{\textbf}

\renewcommand{\phi}{\varphi}

\renewcommand{\tilde}{\widetilde}

\newcommand{\Q}{\mathbb{Q}}
\newcommand{\R}{\mathbb{R}}

\newcommand{\Z}{\mathbb{Z}}
\renewcommand{\P}{\mathbb{P}}

\newcommand{\E}{\mathbb{E}}

\DeclareMathOperator{\Hom}{Hom}

\DeclareMathOperator{\Vol}{Vol}

\DeclareMathOperator{\conv}{conv}
\DeclareMathOperator{\lcm}{lcm}
\DeclareMathOperator{\sqf}{sqf}

\newtheorem{theorem}{Theorem}
\newtheorem{lemma}[theorem]{Lemma}

\newtheorem{corollary}[theorem]{Corollary}

\theoremstyle{definition}

\newtheorem{remark}[theorem]{Remark}
\newtheorem{example}[theorem]{Example}

\newcommand{\excise}[1]{}

\begin{document}

\title[]{
Counting lattice points in free sums of polytopes}

\author{Alan Stapledon}
\email{astapldn@gmail.com}

\date{}
\thanks{This work was initiated and partially completed while  the author was supported by a University of Sydney Postdoctoral Fellowship under the
supervision of A/Prof Anthony Henderson. The author would like to thank Anthony Henderson for some useful discussions.}

\begin{abstract}
We show how to compute the Ehrhart polynomial of the free sum of two lattice polytopes containing the origin $P$ and $Q$  in terms of the enumerative combinatorics of $P$ and $Q$. This generalizes work of Beck, Jayawant, McAllister, and Braun, and follows from the observation that the weighted $h^*$-polynomial is multiplicative with respect to the free sum. We deduce that given a lattice polytope $P$ containing the origin,  the problem of computing the number of lattice points in all rational dilates of $P$ is equivalent to the problem of computing the number of lattice points in all integer dilates of all free sums of $P$ with itself.
\end{abstract}

\maketitle


Let $P$ and $Q$ be full-dimensional
lattice polytopes containing the origin with respect to lattices $N_P \cong \Z^{\dim P}$ and $N_Q \cong \Z^{\dim Q}$ respectively.
The \define{free sum} (also known as `direct sum') $P \oplus Q$ is a full-dimensional lattice polytope
containing the origin in the lattice $N_P \oplus N_Q$, defined by:
\[
P \oplus Q = \conv( (P \times 0_Q) \cup (0_P \times Q) ) \subseteq (N_P \oplus N_Q)_\R,
\]
where $\conv(S)$ denotes the convex hull of a set $S$, $N_\R := N \otimes_\R \R$ for a lattice $N$, and $0_P,0_Q$ denote the origin in $N_P,N_Q$ respectively.

The \define{Ehrhart polynomial} $f(P;m)$ of $P$ is a polynomial of degree $\dim P$ characterized by the property that
$f(P;m) = \# (mP \cap N_P)$ for all $m \in \Z_{\ge 0}$ \cite{ehrhartpolynomial}. Our goal is to describe
the Ehrhart polynomial of $P \oplus Q$
in terms of the
enumerative combinatorics of $P$ and $Q$.

We first observe that $\{ \#(\lambda P \cap N_P)  \mid \lambda \in \Q_{\ge 0} \}$ and $\{ \#(\lambda Q \cap N_Q)  \mid \lambda \in \Q_{\ge 0} \}$
determine $\{ \#(\lambda (P \oplus Q) \cap (N_P \oplus N_Q))  \mid \lambda \in \Q_{\ge 0} \}$, and hence the set $\{ \#(m (P \oplus Q) \cap (N_P \oplus N_Q))  \mid m \in \Z_{\ge 0} \}$, which is encoded by the Ehrhart polynomial of $P \oplus Q$
(see \eqref{e:compare} for a partial converse). Indeed, this follows from the following observation: if $\partial_{\ne 0} P$ denotes the union of the facets of $P$ not containing the origin, then, by
definition, for any $\lambda \in \Q_{\ge 0}$:
\[
\#(\partial_{\ne 0}( \lambda P ) \cap N_P ) = \#(\lambda P \cap N_P) - \max_{0 \le \lambda' < \lambda} \#(\lambda' P \cap N_P),
\]
and
\begin{equation}\label{e:boundary}
\partial_{\ne 0}( \lambda(P \oplus Q)) = \bigcup_{\substack{\lambda_P,\lambda_Q \ge  0 \\ \lambda_P + \lambda_Q = \lambda}}
\partial_{\ne 0}( \lambda_P P ) \times \partial_{\ne 0}( \lambda_Q Q),
\end{equation}
where the right hand side is a disjoint union. 

It will be useful to express the invariants above in terms of corresponding generating series. Firstly, the Ehrhart polynomial may be encoded as follows:
\begin{equation}\label{e:Ehrhart}
\sum_{m \ge 0} f(P;m) t^m = \frac{h^*(P;t)}{(1 - t)^{\dim P + 1}},
\end{equation}
where $h^*(P;t) \in \Z[t]$ is a polynomial of degree at most $\dim P$ with non-negative integer coefficients, called the \define{$h^*$-polynomial} of $P$ \cite{StaDecompositions}. Secondly, let $M_P := \Hom(N_P,\Z)$ be the dual lattice, and recall that
the dual polyhedron $P^{\vee}$ is defined to be
$P^{\vee} = \{ u \in (M_P)_\R \mid \langle u,v \rangle \ge -1 \textrm{ for all } v \in P \}$.
Let
\begin{equation}\label{e:Gorenstein}
r_P := \min \{ r \in \Z_{> 0 }  \mid rP^{\vee} \textrm{ is a lattice polyhedron } \}.
\end{equation}
Note that since $(P \oplus Q)^{\vee}$ is the Cartesian product  $P^{\vee} \times Q^{\vee}$, we have $r_{P \oplus Q} = \lcm(r_P,r_Q)$.
Then
one may associate a generating series encoding $\{ \#(\lambda P \cap N_P)  \mid \lambda \in \Q_{\ge 0} \}$:
\begin{equation}\label{e:altdef}
\sum_{\lambda \in \Q_{\ge 0}} \#(\partial_{\ne 0}( \lambda P ) \cap N_P ) t^{\lambda} = \frac{\tilde{h}(P;t)}{(1 - t)^{\dim P}},
\end{equation}
where $\tilde{h}(P;t) \in \Z[t^{\frac{1}{r_P}}]$ is a polynomial of degree at most $\dim P$ with fractional exponents and non-negative integer coefficients, called the \define{weighted $h^*$-polynomial} of $P$.

\begin{example}
Let $N_P = \Z$ and let $P = [-2,2]$, $Q = [-1,3] = P + 1$. Then $r_P = 2, r_Q = 3$, and
one may compute:
\begin{align*}
\tilde{h}(P;t) &= 1 + 2t^{1/2} + t, \\
\tilde{h}(Q;t) &= 1 + t^{1/3} + t^{2/3} +  t, \\
h^*(P;t) &= h^*(Q;t) = 1 + 3t, \\
\tilde{h}(P \oplus P;t) &= \tilde{h}(P;t)\tilde{h}(P;t)
=  1 + 4t^{1/2} + 6t + 4t^{1/2} + t^{2}, \\
h^*(P \oplus P;t) &= 1 + 10t + 5t^2, \\
\tilde{h}(P \oplus Q;t) &= \tilde{h}(P;t)\tilde{h}(Q;t) \\
&= 1 + t^{1/3} + 2t^{1/2} + t^{2/3}  + 2t^{5/6} +  2t \\
&\hspace{0.5cm} + 2t^{7/6} + t^{4/3} + 2t^{3/2} + t^{5/3}  + t^{2},\\
h^*(P \oplus Q;t) &= 1 + 8t + 7t^2.
\end{align*}

\end{example}

\begin{remark}\label{r:review}
The weighted $h^*$-polynomial was introduced in \cite{Weighted} and generalized in \cite[Section~4.3]{Monodromy}. For the specific definition given in \eqref{e:altdef}, see
  the proof of Proposition 2.6 in \cite{Weighted} with $\lambda \equiv 0$ and $s = t$, and note that, roughly speaking, using the denominator $(1- t)^{\dim P}$ in \eqref{e:altdef} rather than the denominator $(1- t)^{\dim P + 1}$ (cf. \eqref{e:Ehrhart}) corresponds to enumerating lattice points on the boundary of the polytope, rather than those in the polytope itself.  For the non-negativity of the coefficients together with a formula to compute $\tilde{h}(P;t)$, see (15) in \cite{Weighted}.  For the fact that $\tilde{h}(P;t) \in \Z[t^{\frac{1}{r_P}}]$, see Remark~\ref{r:rp} below. Note that it follows from \eqref{e:altdef} that if we write $\tilde{h}(P;t) = \sum_{j \in \Q} \tilde{h}_{P,j} t^j$, then the polynomial $\sum_{i \in \Z} \tilde{h}_{P,i} t^i$ consisting of the terms with  integer-valued exponents of $t$ is precisely the $h^*$-polynomial associated to the lattice polyhedral complex determined by the union of the facets of $P$ not containing the origin.
\end{remark}

Moreover, let \begin{equation}\label{e:Psi} \Psi: \bigcup_{r \in \Z_{> 0}} \R[t^{1/r}] \rightarrow \R[t] \end{equation} denote the $\R$-linear map defined by $\Psi(t^j) = t^{\lceil j \rceil}$ for all $j \in \Q_{\ge 0}$. Then we recover the $h^*$-polynomial of $P$ via the formula (see (14) in \cite{Weighted}):
\begin{equation}\label{e:roundup}
 h^*(P;t) = \Psi(\tilde{h}(P;t)).
 \end{equation}
We also note that when the origin lies in the relative interior of $P$, we have the symmetry \cite[Corollary~2.12]{Weighted}:
\begin{equation}\label{e:symmetry}
\tilde{h}(P;t) = t^{\dim P} \tilde{h}(P;t^{-1}).
\end{equation}

\begin{remark}
The following alternative definition of the weighted $h^*$-polynomial and surrounding discussion is due to Benjamin Nill. Let $Q \subset \R^n$ be a full-dimensional rational polytope and let $k$ be the smallest positive integer such 
that $kQ$ is a lattice polytope. Consider the function defined by $f(Q;m) = \# (mQ \cap \R^n)$ for all $m \in \Z_{\ge 0}$. In \cite{StaDecompositions}, Stanley proved that the associated generating series has the form:
\[
\sum_{m \ge 0}  f(Q;m) t^m = \frac{h^*(Q;t)}{(1 - t^k)^{n + 1}}, 
\]
where $h^*(Q;t)$ is a polynomial of degree at most $k(n + 1) - 1$ with non-negative integer coefficients, called the associated $h^*$-polynomial (see \eqref{e:Ehrhart} for the case when $k = 1$).
If $P$ is a lattice polytope containing the origin, then, with the notation above, one can verify from the definitions that
\[
\tilde{h}(P;t) = h^*(\frac{1}{r_P}P; t^{\frac{1}{r_P}}).
\]
In particular, it follows that the symmetry \eqref{e:symmetry} is equivalent to the statement that if $Q$ is a rational polytope containing the origin in its relative interior such that the associated dual polytope is a lattice polytope, then the coefficients of $h^*(Q;t)$ are symmetric. The latter fact was proved independently by Fiset and Kasprzyk \cite{FKPalindromic}.
\end{remark}

Then \eqref{e:boundary} immediately implies the following multiplicative formula.

\begin{lemma}\label{l:main}
Let $P,Q$ be full-dimensional lattice polytopes containing the origin with respect to lattices $N_P,N_Q$ respectively. Then
\[
\tilde{h}(P \oplus Q;t) = \tilde{h}(P;t)\tilde{h}(Q;t).
\]
\end{lemma}

Combined with \eqref{e:roundup}, we deduce the following formula for the Ehrhart polynomial of $P \oplus Q$:
\begin{equation}\label{e:compute}
h^*(P \oplus Q;t) = \Psi(\tilde{h}(P;t)\tilde{h}(Q;t)).
\end{equation}

\begin{remark}\label{r:exponents}
Let $\Theta: \bigcup_{r \in \Z_{> 0}} \R[t^{1/r}] \rightarrow \R[t]$ denote the $\R$-linear map defined by $\Theta(t^j) = t^{j - \lfloor j \rfloor}$ for all $j \in \Q_{\ge 0}$. Then \cite[Example~4.12]{Monodromy} gives an explicit formula for $\Theta(\tilde{h}(P;t))$ that we will describe below.

Each facet $F$ of $P$ not containing the origin has the form
\[
F = P \cap \{ v \in (N_P)_\R \mid \langle u_F,v \rangle = -m_F \},
\]
where $u_F \in M_P$ is a primitive integer vector, and $m_F \in \Z_{>0}$ is the  \define{lattice distance} of $F$ from the origin.
Then the vertices of $P^{\vee}$ are precisely $\{ \frac{u_F}{m_F} \mid F \textrm{ facet of } P, 0 \notin F \}$, and hence
$r_P = \lcm(m_F \mid  F \textrm { facet of } P, 0 \notin F)$. Then
\[
\Theta(\tilde{h}(P;t)) = \sum_{ \substack{ F \textrm{ facet of } P \\ 0 \notin F } } \Vol(F)  \sum_{i = 0}^{m_F - 1} t^{\frac{i}{m_F}},
\]
where $\Vol(F)$ is defined in Remark~\ref{r:normvol} below.
\end{remark}

\begin{remark}\label{r:normvol}
For any lattice polytope $F$, $h^*(F;1)$ is equal to the \define{normalized volume} $\Vol(F)$ of $F$, i.e. after possibly replacing
the underlying lattice with a smaller lattice, we may assume that $F \subseteq N_\R \cong \R^{\dim F}$ for some lattice $N$, and then
$\Vol(F)$ is $(\dim F)!$ times the Euclidean volume of $F$. In the formula in Remark~\ref{r:exponents}, to make the connection with \cite[Example~4.12]{Monodromy} explicit, observe
that $\Vol(F') = m_F \Vol(F)$, where $F'$ is the convex hull of $F$ and the origin.
\end{remark}

\begin{remark}\label{r:rp}
Remark~\ref{r:exponents} shows that
$r_P$ is the minimal choice of denominator in the fractional exponents in $\tilde{h}(P;t)$ in the sense that $\tilde{h}(P;t) \in \Z[t^{1/r_P}]$ and if $\tilde{h}(P;t) \in \Z[t^{1/r}]$, then $r_P$ divides $r$. 
For example, $\tilde{h}(P;t) = h^*(P;t)$ if and only if $r_P = 1$.
\end{remark}

\begin{remark}
If $P$ and $Q$ contain the origin, but are not full-dimensional, then one may apply the results above after replacing $N_P$ and $N_Q$ by their intersections with the linear spans of $P$ and $Q$ respectively. If $P$ contains the origin but not $Q$, then one may replace $Q$ with $Q' = \conv(Q,0_{Q})$ since $P \oplus Q = P \oplus Q'$.

If neither $P$ nor $Q$ contain the origin, but satisfy the property that the affine spans of $P$ and $Q$ are strict subsets of the linear spans of $P$ and $Q$ respectively, then $P$, $Q$ and $P \oplus Q$ are the unique facets not containing the origin of $P' = \conv(P,0_{P})$, $Q' = \conv(Q,0_{Q})$ and $P' \oplus Q'$ respectively. In this case, by Remark~\ref{r:review} and Lemma~\ref{l:main}, $h^*(P \oplus Q; t)$ is the polynomial
consisting of the terms of $\tilde{h}(P' \oplus Q'; t) = \tilde{h}(P';t)\tilde{h}(Q';t)$ with  integer-valued exponents of $t$.

\end{remark}

We deduce a new proof to the following result of Beck, Jayawant and McAllister \cite[Theorem~1.3]{BJMLattice}, which itself generalizes a result of Braun \cite{BraEhrhart}.

\begin{corollary}
Let $P,Q$ be full-dimensional lattice polytopes containing the origin with respect to lattices $N_P,N_Q$ respectively. Then
\[
h^*(P \oplus Q;t) = h^*(P;t)h^*(Q;t) \iff r_P = 1 \textrm{ or } r_Q = 1.
\]
\end{corollary}
\begin{proof}
If we write $\tilde{h}(P;t) = \sum_{j \in \Q} \tilde{h}_{P,j} t^j$, then by \eqref{e:roundup} and \eqref{e:compute},
\[
h^*(P \oplus Q;t) = \sum_{j,j' \in \Q} \tilde{h}_{P,j} \tilde{h}_{Q,j'} t^{\lceil j + j' \rceil},
\]
\[
h^*(P;t)h^*(Q;t) = \sum_{j,j' \in \Q} \tilde{h}_{P,j} \tilde{h}_{Q,j'} t^{\lceil j \rceil + \lceil j' \rceil},
\]
If $r_P = 1$ or $r_Q = 1$, then we have equality. If $r_P,r_Q > 1$, then by Remark~\ref{r:exponents}, there exists $(j,j') \in \Q^2$ such that
$\tilde{h}_{P,j},\tilde{h}_{Q,j'} > 0$ and $0 < j - \lfloor j \rfloor, j' - \lfloor j' \rfloor \le 1/2$. Then $\lceil j \rceil + \lceil j' \rceil = \lceil j + j' \rceil + 1$, and the non-negativity of the coefficients of $\tilde{h}(P;t)$ and $\tilde{h}(Q;t)$ implies that $h^*(P \oplus Q;t) \ne h^*(P;t)h^*(Q;t)$.
\end{proof}

\begin{remark}
A lattice polytope $P$ satisfying $r_P = 1$ and containing the origin in its relative interior is called \define{reflexive}.  These polytopes have received a lot of attention as, geometrically, they correspond to Fano toric varieties.  In particular, they play a central  
role in Batyrev and Borisov's construction of mirror pairs of Calabi-Yau varieties \cite{BBMirror}.
\end{remark}

\begin{remark}
The weighted $h^*$-polynomial arises naturally in two distinct geometric situations: Firstly, in the computation of dimensions of the graded pieces of orbifold cohomology groups of toric stacks \cite[Theorem~4.3]{Weighted} and, more generally, in the computation of motivic integrals on
toric stacks \cite[Theorem~6.5]{Motivic}. Secondly, in computations of the action of monodromy on the cohomology of the fiber of a degeneration of complex hypersurfaces (or the associated Milnor fiber) \cite[Sections 5,6, Corollary~5.12]{Monodromy}. In particular, taking the free sum of polytopes above corresponds to taking products of the associated toric stacks, and the multiplicative formula in Lemma~\ref{l:main} may be viewed as a K\"unneth formula for the dimensions of the graded pieces of orbifold cohomology groups of toric stacks.
\end{remark}

\begin{example}
In order to provide a wider class of examples of weighted $h^*$-polynomials, we consider a class of examples of lattice polytopes used by Payne in  \cite{PayEhrhart}.
Consider positive integers $\alpha_0 \ge \alpha_{1} \ge \cdots \ge \alpha_d$ with no common factor and let $N = \Z^{d + 1}/(\sum_{i = 0}^{d} \alpha_{i} e_{i} = 0)$, where $e_0, \ldots, e_d$ denotes the standard basis of $\Z^{d + 1}$.  Observe that $N$ is a  lattice  of
rank $d$ and, if
$P(\alpha_{0}, \ldots, \alpha_{d})$ denotes the convex hull of the images of  $e_{0}, \ldots, e_{d}$, then $P(\alpha_{0}, \ldots, \alpha_{d})$ is a lattice polytope containing the origin in its relative interior.
The following formula follows from the proof of \cite[Lemma~9.1]{Additive}:
\[
\tilde{h}(P(\alpha_{0}, \ldots, \alpha_{d});t) =
\sum_{i = 0}^{d} \sum_{j = 0}^{\alpha_i - 1}
t^{       \sum_{0 \le k \le d, k \ne i}     (\frac{j \alpha_k}{\alpha_i} - \lfloor \frac{j \alpha_k}{\alpha_i} \rfloor    )   +  \sum_{k = i + 1}^d \phi( \frac{j \alpha_k}{\alpha_i})     },
\]
where $\phi(x) = 1$ if $x$ is an integer and $\phi(x) = 0$, otherwise. 
\end{example}

We now consider a partial converse to \eqref{e:compute}.
We will use the following lemma due to Terence Harris \cite{THSums}.

\begin{lemma}
Let $f(t) \in \R[t^{1/r}]$ be a polynomial with non-negative coefficients and fractional exponents for some $r \in \Z_{> 0}$. Fix a positive real number $x$. For any $n \in \Z_{> 0}$, let $f^*_n(t) := \Psi(f(t)^n) \in \R[t]$, where $\Psi$ is defined in \eqref{e:Psi} i.e. $f^*_n(t)$ is obtained from  $f(t)^n$ by rounding up exponents in $t$. Then
\[
f^*_n(x)  \le f(x)^n \le x^{\frac{1}{r} - 1} f^*_n(x)  \textrm{ if } 0 < x \le 1,
\]
\[
f(x)^n \le f^*_n(x) \le x^{1 - \frac{1}{r}} f(x)^n \textrm{ if } x \ge 1.
\]
In particular, given any polynomial $f(t) \in \bigcup_{r \in \Z_{> 0}} \R[t^{1/r}]$ with non-negative coefficients and fractional exponents,
\[
f(x) = \lim_{n \rightarrow \infty} f^*_n(x)^{1/n},
\]
and $f(t)$ determines and is determined by $\{ f^*_n(t) \mid n \in \Z_{> 0} \}$.
\end{lemma}
\begin{proof}
First assume that $x \ge 1$. When $n = 1$, the inequalities  $\Psi(f(t))_{t = x}  \le f(x) \le x^{1 - \frac{1}{r}} \Psi(f(t))_{t = x}$ follow from
the fact that $x^{\frac{i}{r}} \le x^{\lceil \frac{i}{r} \rceil } \le x^{1 - \frac{1}{r}}x^{\frac{i}{r}}$ for any $i \in \Z_{\ge 0}$, and the assumption that the coefficients of $f(t)$ are non-negative. When $n \ge 1$, the inequalities follow by replacing $f(t)$ with $f(t)^n$. When $0 < x \le 1$, $x^{\lceil \frac{i}{r} \rceil } \le x^{\frac{i}{r}}  \le x^{\frac{1}{r} - 1}x^{\lceil \frac{i}{r} \rceil }$ and the result follows similarly. The final statement follows immediately.
\end{proof}

For any positive integer $n$, let $P^{\oplus n}$ denote the free sum of $P$ with itself $n$ times.
By Lemma~\ref{l:main} and \eqref{e:roundup}, one may apply the above lemma with $f(t) = \tilde{h}(P;t)$, $f(t)^n = \tilde{h}(P^{\oplus n};t)$, $f_n^*(t) = h^*(P^{\oplus n};t)$ and $r = r_P$, to obtain the corollary below.

\begin{corollary}\label{c:asymptotics}
Let $P$ be a full-dimensional lattice polytope containing the origin with respect to a lattice $N_P$.
Fix a positive real number $x$. For any $n \in \Z_{> 0}$, and with $r_P$ as defined in \eqref{e:Gorenstein},
\[
h^*(P^{\oplus n};x)  \le \tilde{h}(P;x)^n \le x^{\frac{1}{r_P} - 1} h^*(P^{\oplus n};x)  \textrm{ if } 0 < x \le 1,
\]
\[
\tilde{h}(P;x)^n \le h^*(P^{\oplus n};x) \le x^{1 - \frac{1}{r_P}} \tilde{h}(P;x)^n \textrm{ if } x \ge 1.
\]
In particular,
\[
\tilde{h}(P;x) = \lim_{n \rightarrow \infty} h^*(P^{\oplus n};x)^{1/n},
\]
and $\tilde{h}(P;t)$ determines and is determined by $\{ h^*(P^{\oplus n};t) \mid n \in \Z_{> 0} \}$.
\end{corollary}

Note that the final statement above states that the following two sets contain precisely the same information:
\begin{equation}\label{e:compare} \{ \#(\lambda P \cap N_P)  \mid \lambda \in \Q_{\ge 0} \}, \end{equation}
\[
\{ \#(m P^{\oplus n} \cap (N_P \oplus \cdots \oplus N_P))  \mid m \in \Z_{\ge 0}, n \in \Z_{> 0} \}.
\]

\begin{remark}\label{r:subsequence}
From the proof of Corollary~\ref{c:asymptotics},
$\tilde{h}(P;t)$ determines and is determined by $\{ h^*(P^{\oplus n};t) \mid n \in S \}$ for any infinite subset $S \subseteq \Z_{> 0}$.
\end{remark}

Finally, the above results together with the central limit theorem describe some of the asymptotic behavior of $h^*(P^{\oplus n};t)$ as $n \rightarrow \infty$.
More precisely, let $X^*_P$, $\tilde{X}_P$ be $\R$-valued random variables with probability distributions on $\R$ defined by:
\[
\P( X^*_P = i ) = \frac{h^*_{P,i}}{\Vol(P)},
\]
\[
\P( \tilde{X}_P = j) = \frac{\tilde{h}_{P,j}}{\Vol(P)},
\]
where $h^*(P;t)  =  \sum_{i \in \Z} h^*_{P,i} t^i$ and $\tilde{h}(P;t)  =  \sum_{j \in \Q} \tilde{h}_{P,j} t^j$, and $h^*(P;1) = \tilde{h}(P;1) = \Vol(P)$ (see Remark~\ref{r:normvol}). Equivalently, the
moment generating functions of $X^*_P$ and $\tilde{X}_P$ are given by:
\[
\E[e^{sX^*_P}] = \frac{1}{\Vol(P)} h^*(P;e^s),\]\[\E[e^{s\tilde{X}_P}] = \frac{1}{\Vol(P)} \tilde{h}(P;e^s).
\]
Let $\tilde{\mu}_P$ and $\tilde{\sigma}_P$ denote the mean and standard deviation of $\tilde{X}_P$ respectively, and let $\mathcal{N}(\mu,\sigma)$ denote the normal distribution
with mean $\mu$ and variance $\sigma$.

\begin{example}\label{e:meanhalfdim}
When the
origin lies in the relative interior of $P$, \eqref{e:symmetry} implies that $\tilde{\mu}_P  = \frac{\dim P}{2}$.
\end{example}

\begin{corollary}\label{c:clt}
Let $P$ be a full-dimensional lattice polytope containing the origin in  a lattice $N_P$. 
Then
\[
\frac{ X^*_{P^{\oplus n}} - n\tilde{\mu}_P}{\sqrt{n}}  \overset{d}{\rightarrow} \mathcal{N}(0,\tilde{\sigma}_P),
\]
as $n \rightarrow \infty$, where convergence means convergence in distribution (see Remark~\ref{r:convergence}). 
\end{corollary}
\begin{proof}
Fix $s \in \R$. Then by Corollary~\ref{c:asymptotics}, for any $n \in \Z_{> 0}$,
\[
e^{-s \tilde{\mu}_P \sqrt{n}} h^*(P^{\oplus n};e^{\frac{s}{\sqrt{n}}})  \le e^{-s \tilde{\mu}_P \sqrt{n}} \tilde{h}(P;e^{\frac{s}{\sqrt{n}}})^n \le  e^{\frac{s(1 - r_P)}{\sqrt{n}r_P}} e^{-s \tilde{\mu}_P \sqrt{n}} h^*(P^{\oplus n};e^{\frac{s}{\sqrt{n}}})  \textrm{ if } s \le 0,
\]
\[
e^{-s \tilde{\mu}_P \sqrt{n}} \tilde{h}(P;e^{\frac{s}{\sqrt{n}}})^n \le e^{-s \tilde{\mu}_P \sqrt{n}} h^*(P^{\oplus n};e^{\frac{s}{\sqrt{n}}}) \le  e^{\frac{s(r_P - 1)}{\sqrt{n}r_P}} e^{-s \tilde{\mu}_P \sqrt{n}} \tilde{h}(P;e^{\frac{s}{\sqrt{n}}})^n  \textrm{ if } s \ge 0.
\]
If $\tilde{X}_1,\ldots,\tilde{X}_n$ are iid random variables with distribution $\tilde{X}_P$, and $\tilde{Z}_n := \frac{(\tilde{X}_1 - \tilde{\mu}_P) + \cdots + (\tilde{X}_n - \tilde{\mu}_P)}{\sqrt{n}}$, then either a direct computation or invoking the central limit theorem gives:
\[
\lim_{n \rightarrow \infty} \E[e^{s\tilde{Z}_n}] = \lim_{n \rightarrow \infty} e^{-s \tilde{\mu}_P \sqrt{n}} \tilde{h}(P;e^{\frac{s}{\sqrt{n}}})^n = e^{\frac{(s \tilde{\sigma}_P)^2}{2}},
\]
where $e^{\frac{(s \tilde{\sigma}_P)^2}{2}}$ is the moment generating function of $\mathcal{N}(0,\tilde{\sigma}_P)$.
If $Z_n^* := \frac{X^*_{P^{\oplus n}} - n\tilde{\mu}_P}{\sqrt{n}}$, then the above inequalities state that
\[
\E[e^{sZ_n^*}]  \le \E[e^{s\tilde{Z}_n}] \le  e^{\frac{s(1 - r_P)}{\sqrt{n}r_P}} \E[e^{sZ_n^*}]  \textrm{ if } s \le 0,
\]
\[
\E[e^{s\tilde{Z}_n}] \le \E[e^{sZ_n^*}] \le  e^{\frac{s(r_P - 1)}{\sqrt{n}r_P}} \E[e^{s\tilde{Z}_n}] \textrm{ if } s \ge 0.
\]
Hence $\lim_{n \rightarrow \infty} \E[e^{sZ_n^*}] = \lim_{n \rightarrow \infty} \E[e^{s\tilde{Z}_n}] = e^{\frac{(s \tilde{\sigma}_P)^2}{2}}$ and the result follows since convergence of the moment generating functions of $Z_n^*$ to the moment generating function of $\mathcal{N}(0,\tilde{\sigma}_P)$  implies convergence of the corresponding distributions \cite[Theorem 3]{CurNote} (note that all moment generating functions above converge for all $s \in \R$).
\end{proof}

\begin{remark}\label{r:convergence}
The convergence in Corollary~\ref{c:clt} is defined in terms of the corresponding cumulative distribution functions as follows: for all $x \in \R$, if we write
 $h^*(P^{\oplus n};t) = \sum_{i \in \Z}  h^*_{P^{\oplus n},_i}t^i$,
\[
F_n(x) = \P\big(\frac{ X^*_{P^{\oplus n}} - n\tilde{\mu}_P}{\sqrt{n}} \le x\big) = \frac{1}{\Vol(P)^n}\sum_{ \substack{i \in \Z \\ i \le \sqrt{n} x + n\tilde{\mu}_P}} h^*_{P^{\oplus n},_i},
\]
and $\Phi_{\tilde{\sigma}_P}(x) = \frac{1}{\sqrt{2\pi}\tilde{\sigma}_P} \int_{- \infty}^{x} e^{-(\frac{s}{\tilde{\sigma}_P })^2/2} ds$, then $\underset{n \rightarrow \infty}{\lim} F_n(x)  = \Phi_{\tilde{\sigma}_P}(x)$.
\end{remark}

\begin{example}
A lattice polytope $P$ containing the origin is a \define{standard simplex} if its non-zero vertices form a basis of $N_P$.
In this case, $\tilde{h}(P;t) = h^*(P;t) = 1$, $\tilde{\mu}_P = \tilde{\sigma}_P = 0$ and $P^{\oplus n}$ is a standard simplex for all $n$.
\end{example}

\begin{example}
Fix a positive integer $n$ and consider the lattice $N = \Z[e^{\frac{2\pi \sqrt{-1}}{n}}]$. The \define{n-th cyclotomic polytope} $\mathcal{C}_n$ is the convex hull
of all $n$-th roots of unity in $N \otimes_\Z \R \cong \R^{\phi(n)}$, where $\phi$ is the Euler totient function. In \cite[Theorem~7,Lemma~8,Corollary~9]{BSCyclotomic}, Beck and
Ho{\c{s}}ten prove that the lattice points of $\mathcal{C}_n$ consist of the $n$-th roots of unity, which are vertices, together with the origin, which is the unique interior lattice point,
and they identify $\mathcal{C}_n$ with $\mathcal{C}_{\sqf(n)}^{\oplus \frac{n}{\sqf(n)}}$, where $\sqf(n)$ denotes the square-free part of $n$ i.e. the product of the prime divisors of $n$. Moreover, they prove that $\mathcal{C}_n$ is reflexive if $n$ is divisible by at most two odd primes, and they show how to compute $h^*(\mathcal{C}_n;t) = \tilde{h}(\mathcal{C}_n;t)$ for $n \le 104$. The smallest value of $n$ for which $h^*(\mathcal{C}_n;t)$ is unknown is $n = 105 = 3 \cdot 5 \cdot 7$. We refer the reader to \cite{BSCyclotomic} for further results and details.

By \eqref{e:compute}, and using Beck and
Ho{\c{s}}ten's result above, for any positive integer $n$,
\[
h^*(\mathcal{C}_n;t) =     h^*(\mathcal{C}_{\sqf(n)}^{\oplus \frac{n}{\sqf(n)}};t) =  \Psi(\tilde{h}(\mathcal{C}_{\sqf(n)};t)^{\frac{n}{\sqf(n)}}).
\]
It follows from this observation and Remark~\ref{r:subsequence} that for any product of distinct primes $a$, the problem of computing
 $\{ h^*(\mathcal{C}_n;t) \mid \sqf(n) = a \}$ is equivalent to the problem of computing $\tilde{h}(\mathcal{C}_a;t)$.
More precisely, consider a strictly increasing sequence of positive integers $\{ n_k \}_{k \in \Z_{> 0}}$ 
satisfying $\sqf(n_k) = a$ for all $k$. Then by Remark~\ref{r:subsequence}, 
$\{ h^*(\mathcal{C}_{n_k};t) = h^*(\mathcal{C}_{a}^{\oplus \frac{n_k}{a}};t) \mid k \in \Z_{> 0} \}$ determines and is determined by $\tilde{h}(\mathcal{C}_{a};t)$. Moreover, since
$\tilde{\mu}_{\mathcal{C}_a} = \frac{\dim \mathcal{C}_a}{2} = \frac{\phi(a)}{2}$ by Example~\ref{e:meanhalfdim}, setting $P = \mathcal{C}_{a}$ and $n = \frac{n_k}{a}$ in Corollary~\ref{c:clt}  implies that
\[
\frac{ X^*_{\mathcal{C}_{n_k}} - \frac{n_k \phi(a)}{2a}}{\sqrt{\frac{n_k}{a}}}  \overset{d}{\rightarrow} \mathcal{N}(0,\tilde{\sigma}_{\mathcal{C}_a}),
\]
as $k \rightarrow \infty$. We note that it is an open problem to compute $\tilde{\sigma}_{\mathcal{C}_a}$ for any product of distinct primes $a$.

\end{example}

\bibliographystyle{amsplain}

\begin{thebibliography}{99}






\bibitem{BBMirror}
Victor Batyrev and Lev Borisov,   \emph{Mirror duality and string-theoretic Hodge numbers},
Invent. Math. \textbf{126} (1996), no. 1, 183-203.


\bibitem{BSCyclotomic}
Matthias Beck and Serkan Ho{\c{s}}ten, \emph{Cyclotomic polytopes and growth series of cyclotomic lattices},
Math. Res. Lett. \textbf{13} (2006), no. 4,  1073-2780.

\bibitem{BJMLattice}
Matthias Beck, Pallavi Jayawant and  Tyrrell McAllister, \emph{Lattice-point generating functions for free sums of convex sets},
J. Combin. Theory Ser. A \textbf{120} (2013), no. 6,  1246--1262.



%



%

%

\bibitem{BraEhrhart}
Benjamin Braun, \emph{An {E}hrhart series formula for reflexive polytopes},
Electron. J. Combin. \textbf{13} (2006), no. 13, 1077-8926.




%


%

\bibitem{CurNote}
J. H. Curtis, \emph{A note on the theory of moment generating functions},
Annals of Math. Statistics \textbf{13} (1942), 430--433.


%


%


%

%

%


\bibitem{ehrhartpolynomial}
Eug{\`e}ne Ehrhart, \emph{Sur les poly\`edres rationnels homoth\'etiques \`a
  {$n$}\ dimensions}, C. R. Acad. Sci. Paris \textbf{254} (1962), 616--618.

%




\bibitem{FKPalindromic}
Matthew Fiset and Alexander Kasprzyk, \emph{A note on palindromic $\delta$-vectors for
certain rational polytopes}, Electron. J. Combin. \textbf{15} (2008), no. 1, Note 18.


%

\bibitem{THSums}
Terence Harris, Sydney University Mathematical Society Problem Competition 2015, Solution to Question 8,
http://www.maths.usyd.edu.au/u/SUMS/sols2015.pdf.

%

%


%

%


%

%



%

%




%




%


%








\bibitem{PayEhrhart}
Sam Payne, \emph{Ehrhart series and lattice triangulations},
Discrete Comput. Geom. \textbf{40} (2008), no. 3, 365--376.



%
%

%




%

%
\bibitem{StaDecompositions}
Richard Stanley, \emph{Decompositions of rational convex polytopes}, Ann. Discrete Math.  \textbf{6} (1980), 333--342.



%


%
\bibitem{Weighted}
Alan Stapledon, \emph{Weighted {E}hrhart theory and orbifold cohomology}, Adv. Math. \textbf{219} (2008), no. 1, 63--88.

%
\bibitem{Motivic}
\bysame,
\emph{Motivic integration on toric stacks}, Comm. Alg. \textbf{37} (2009), no. 11, 3943--3965.

%


\bibitem{Additive}
\bysame,
\emph{Additive number theory and inequalities in {E}hrhart theory},
Int. Math. Res. Not. \textbf{5} (2016), 1497--1540.




\bibitem{Monodromy}
\bysame,
\emph{Formulas for monodromy}, Res. Math. Sci. \textbf{4} (2017), no. 1. 
%


%

%




\end{thebibliography}
\def\cprime{$'$}
\providecommand{\bysame}{\leavevmode\hbox to3em{\hrulefill}\thinspace}
\providecommand{\MR}{\relax\ifhmode\unskip\space\fi MR }
\providecommand{\MRhref}[2]{%
  \href{http://www.ams.org/mathscinet-getitem?mr=#1}{#2}
}
\providecommand{\href}[2]{#2}

\end{document}